\documentclass[12pt, reqno]{amsart}

\usepackage{amsmath,amssymb,amscd,amsxtra,upref}
\usepackage{latexsym}
\usepackage[dvips]{graphics,epsfig}
\usepackage{enumerate}
\usepackage[colorlinks=true, pdfstartview=FitV, linkcolor=blue, citecolor=blue, urlcolor=blue]{hyperref}

\newtheorem{theorem}{Theorem}
\newtheorem{lemma}[theorem]{Lemma}

\newtheorem*{theorema*}{Theorem A}
\newtheorem*{theoremb*}{Theorem B}
\newtheorem*{theoremc*}{Theorem C}
\newtheorem*{theoremd*}{Theorem D}
\newtheorem*{theoreme*}{Theorem E}
\newtheorem*{theoremf*}{Theorem F}
\newtheorem*{theoremg*}{Theorem G}
\newtheorem*{corollaryf*}{Corollary F}

\theoremstyle{remark}

\DeclareMathOperator*{\iintt}{\iint}

\newcommand{\I}{\hspace{0.5mm}\text{I}\hspace{0.5mm}}
\newcommand{\II}{\text{I \hspace{-2.8mm} I} }

\newcommand{\re}{\mathbb{R}}
\newcommand{\rn}{\mathbb{R}^n}

\newcommand{\Dl}{\Delta}

\newcommand{\s}{\sigma}
\newcommand{\wt}{\widetilde}
\newcommand{\dd}{\partial}
\newcommand{\les}{\lesssim}

\newcommand{\al}{\alpha}
\newcommand{\be}{\beta}
\newcommand{\dl}{\delta}
\newcommand{\g}{\gamma}
\newcommand{\jb}[1]
{\langle #1 \rangle}

\newcommand{\R}{\mathbb{R}}

\newcommand{\noi}{\noindent}
\DeclareMathOperator*{\supp}{supp}

\newcommand{\cz}{Calder\'on-Zygmund\ }

\begin{document}

\baselineskip = 15 pt



\address{\'Arp\'ad B\'enyi, Department of Mathematics,
516 High St, Western Washington University, Bellingham, WA 98225,
USA} \email{arpad.benyi@wwu.edu}

\address{Tadahiro Oh, Department of Mathematics, Princeton University, Fine Hall, Washington Rd, Princeton, NJ 08544-1000, USA}
\email{hirooh@math.princeton.edu}

\title[Smoothing of bilinear commutators]
{Smoothing of commutators for a H\"ormander class of bilinear pseudodifferential operators}

\author[\'A. B\'enyi and T. Oh]
{\'Arp\'ad B\'enyi and Tadahiro Oh}

\thanks{The first author is partially supported by Simons Foundation Grant No.~246024. The second author acknowledges support from an AMS-Simons Travel Grant.}


\begin{abstract}
Commutators of bilinear pseudodifferential operators with symbols in the H\"ormander class $BS_{1, 0}^1$ and multiplication by Lipschitz functions are shown to be bilinear Calder\'on-Zygmund operators. A connection with a notion of compactness in
the bilinear setting  for the iteration of the commutators is also made. 
\end{abstract}

\maketitle

\section{Motivation, preliminaries and statements of main results }\label{intro}

The work of Calder\'on and Zygmund on singular integrals and Calder\'on's ideas \cite{Cal1, Cal2} about improving a pseudodifferential calculus, where the smoothness assumptions on the coefficients are minimal, have greatly affected research in quasilinear and nonlinear PDEs. The subsequent investigations about multilinear operators initiated by Coifman and Meyer \cite{CM} in the late 70s have added to the success of Calder\'on's  work on commutators. A classical bilinear estimate, the so-called Kato-Ponce commutator estimate \cite{KP}, is crucial in the study of the Navier-Stokes equations. This estimate is a general Leibniz-type rule which takes the form
\begin{equation}\label{Kato-Ponce}
\|D^\alpha (fg)\|_{L^r}\lesssim \|D^\alpha f\|_{L^p}\|g\|_{L^q}+\|f\|_{L^p}\|D^\alpha g\|_{L^q},
\end{equation}
for \(1<p, q\leq\infty, 0<r<\infty, 1/p+1/q=1/r\) and \(\alpha > 0\). More general Leibniz-type rules that apply to bilinear pseudodifferential operators with symbols in the bilinear
H\"ormander classes \(BS_{\rho, \delta}^m\) (see \eqref{bi-pseudo} below for their definition) can be found, for example, in the works of B\'enyi et al.~\cite{BMNT, BBMNT, BNT} and Bernicot et al.~\cite{BMMN}. Interestingly, for $\alpha=1$ and in dimension one, the Kato-Ponce estimate \eqref{Kato-Ponce} is closely related to the boundedness of the so-called \emph{Calder\'on's first commutator}.
Given a Lipschitz function $a$ and \(f\in L^2\),
define $C(a, f)$ by
\[C (a, f)= p.v.\int_{\re} \frac{a(x)-a(y)}{(x-y)^2}f(y)\,dy.\]

\noi
Then, denoting by $H$ the classical Hilbert transform, we can identify the operator \(C (a, \cdot)\) with the commutator of \(T=H \circ \partial_x \) and the multiplication by the Lipschitz function $a$; that is, \(C (a, f) = [T, a] (f) := T(af)-aT(f).\)
While we have no hope of controlling each of the individual terms defining \([T, a]\), the commutator itself does behave nicely; Calder\'on showed \cite{Cal2} that \(\big\|[T, a]\big\|_{L^2}\leq \|a'\|_{L^\infty}\|f\|_{L^2}\), effectively producing the bilinear boundedness of the operator \(C: \text{Lip}_1\times L^2\to L^2\). Moreover, the boundedness of the first commutator can be extended to give the following result, see \cite[Theorem 4 on p.~90]{MC}:

\begin{theorema*}
Let \(T_\sigma\) be a linear pseudodifferential operator with symbol \(\sigma\in S_{1, 0}^1\)
and $a$ be a Lipschitz function such that \(\nabla a\in L^\infty\).
Then,   \([T_\sigma, a]\) is a linear Calder\'on-Zygmund operator. In particular,  \([T_\sigma, a]\) is bounded on \(L^p, 1<p<\infty\). Conversely, if $[D_j , a]$ is bounded on $L^2$, $j = 1\dots, n$, then \(\nabla a\in L^\infty\).
\end{theorema*}

The statement of Theorem~A is the very manifestation of the so-called \emph{commutator smoothing effect}: while the H\"ormander class of symbols \(S_{1, 0}^1\) does not yield bounded pseudodifferential operators on $L^p$, the commutator with a sufficiently smooth function (Lipschitz in our case) fixes this issue. An application of this result can be found in the work of Kenig, Ponce and Vega \cite{KPV} on nonlinear Schr\"odinger equations.

The smoothing effect of commutators gets better when we commute with special multiplicative functions. For example, the result of Coifman, Rochberg and Weiss \cite{CRW} gives the boundedness on $L^p (\mathbb{R}^n)$, $1 < p < \infty$,  of linear commutators of Calder\'on-Zygmund operators and pointwise multiplication, when the multiplicative function (or symbol) is
in the John-Nirenberg space $BMO$.
Uchiyama \cite{Uch} improved the boundedness to compactness
if the multiplicative function is in $CMO$; here, $CMO$ denotes the  closure of $C^\infty$-functions with compact supports under the $BMO$-norm. The $CMO$ in our context stands
for ``continuous mean oscillation'' and is not to be confused with other versions of  $CMO$ (such as ``central mean oscillation''). In fact, the $CMO$ we are considering coincides with $VMO$, the space of functions of ``vanishing mean oscillation'' studied  by Coifman and Weiss in \cite{CW}, but also differs from other versions of $VMO$ found in the literature; see, for example, \cite{BT1} for further comments on the relation between $CMO$ and $VMO$. An application of this compactness to deriving a Fredholm alternative for equations with $CMO$ coefficients in all $L^p$ spaces with $1<p<\infty$ was given by Iwaniec and Sbordone \cite{IS}. Other important applications appear in the theory of compensated compactness of Coifman, Lions, Meyer and Semmes \cite{CLMS} and in  the integrability theory of Jacobians, see Iwaniec \cite{Iwa}.

In this work,  we seek to extend such results for linear commutators to the multilinear setting. For ease of notation and comprehension, we restrict ourselves to the bilinear case. The bilinear Calder\'on-Zygmund theory is nowadays well understood; for example, the work of Grafakos and Torres \cite{GT} makes available a bilinear \(T(1)\) theorem for such operators. As an application of their \(T(1)\) result,  we can obtain the boundedness of bilinear pseudodifferential operators with symbols in appropriate H\"ormander classes of bilinear pseudodifferential symbols. Moreover, the bilinear H\"ormander pseudodifferential theory has nowadays a similarly solid foundation, see again \cite{BMNT, BBMNT, BNT} and the work of B\'enyi and Torres \cite{BT}.

Our discussion on 
the study of such classes of bilinear operators, on the one hand,
 exploits the characteristics of their kernels in the spatial domain
 and, on the other hand,  makes use of the properties of their symbols in the frequency domain.
 First, consider bilinear operators a priori defined from \(\mathcal S\times\mathcal S\) into \(\mathcal S'\) of the form
\begin{equation}
T (f, g)(x)=\int_{\rn}\int_{\rn} K(x, y, z)f(y)g(z)\,dydz.
\label{T1}
\end{equation}

\noi
Here, we assume that, away from the diagonal \(\Omega= \{(x,y,z) \in \re^{3n}: x=y=z \}\),
 the distributional kernel \(K\) coincides with a function  \(K(x,y,z)\) locally integrable in \(\re^{3n}\setminus \Omega\) satisfying  the following size and regularity conditions in $\mathbb R^{3n}\setminus \Omega$:
\begin{equation}\label{cz-size}
|K(x,y,z)| \lesssim \big(|x-y| + |x-z| + |y-z|\big)^{-2n},
\end{equation}
 and
\begin{equation}\label{cz-regularity}
|K(x,y,z) - K(x',y,z)| \lesssim \frac{|x-x'|}{\big(|x-y| +
|x-z| + |y-z|\big)^{2n + 1}},
\end{equation}
whenever \(|x - x'| \leq \frac{1}{2} \max\{|x-y|,|x-z|\}\). While the condition ~\eqref{cz-regularity} is not the most general that one can impose in such theory, see \cite{GT}, we prefer to work with this simplified formulation in order to avoid unnecessary further technicalities. For symmetry and interpolation purposes we also require that the formal transpose kernels $K^{*1}, K^{*2}$ (of the transpose operators $T^{*1}, T^{*2}$, respectively), given by
\[K^{*1}(x,y,z) = K(y,x,z) \quad \text{and}\quad
K^{*2}(x,y,z)=K(z,y,x),\]

\noi
also satisfy \eqref{cz-regularity}. Moreover, for an additional simplification, in the following we will replace the regularity conditions \eqref{cz-regularity} on \(K, K^{*1}\) and \(K^{*2}\) with the natural conditions on the gradient \(\nabla K\):
\begin{equation} \label{CZK}
|\nabla K(x,y,z)| \lesssim \big(|x-y| +
|x-z| + |y-z|\big)^{-2n -1},
\end{equation}

\noi
for $(x, y, z)\in \R^{3n} \setminus \Omega$.
 We say that such a kernel \(K(x,y,z)\) is a \emph{bilinear \cz kernel}.
Moreover, given  a bilinear operator $T$ defined  in \eqref{T1}
with a \cz  kernel $K$ (which satisfies ~\eqref{cz-size} and ~\eqref{CZK}),
 we say that  \(T\) is a {\it bilinear  \cz operator} if it extends to a bounded operator from \(L^{p_0} \times L^{q_0}\) into \(L^{r_0}\) for {\it some} \(1<p_0, q_0<\infty\) and \(1/p_0+1/q_0=1/r_0 \leq 1\).

The crux of bilinear \cz theory is the following statement, see \cite{GT}.

\begin{theoremb*}
Let \(T\) be a bilinear \cz operator.
Then,   \(T\) maps \(L^p\times L^q\) into \(L^r\) for all \(p, q, r\) such that \(1<p, q<\infty\) and \(1/p + 1/q = 1/r \leq 1\). Moreover, we also have the following end-point boundedness results:
\begin{enumerate}[{\rm (a)}]
\item When \(p=1\) or \(q=1\), then \(T\) maps \(L^p\times L^q\) into \(L^{r, \infty}\);
\item When \(p=q=\infty\), then \(T\) maps \(L^\infty\times L^\infty\) into \(BMO\).
\end{enumerate}
\end{theoremb*}

\noi
Theorem~B assumes the boundedness $L^{p_0}\times L^{q_0}\to L^{r_0}$ of the operator $T$ for some H\"older triple $(p_0, q_0, r_0)$. Obtaining one such boundedness via appropriate cancelation conditions is another topic of interest in the theory of linear and multilinear operators with Calder\'on-Zygmund kernels. A satisfactory answer is provided by the $T(1)$ theorem; the following bilinear version, as stated by Hart \cite{H}, is equivalent to the formulation in \cite{GT} and is strongly influenced by the fundamental work of David and Journ\'e \cite{DJ} in the linear case.

\begin{theoremc*}
Let $T: \mathcal S\times\mathcal S\to \mathcal S'$ be a bilinear singular integral operator with Calder\'on-Zygmund kernel $K$. Then,  $T$ can be extended to a bounded operator from \(L^{p_0} \times L^{q_0}\) into \(L^{r_0}\) for some \(1<p_0, q_0<\infty\) and \(1/p_0+1/q_0=1/r_0 \leq 1\) if and only if
$T$ satisfies the following two conditions:
\begin{itemize}
\item[\textup{(i)}] $T$ has the weak boundedness property,
\item[\textup{(ii)}] $T(1, 1), T^{*1}(1, 1)$ and  $T^{*2}(1, 1)$ are in $BMO$.
\end{itemize}
\end{theoremc*}

\noi
For the definition of the weak boundedness property, see Subsection \ref{subsec:5}.

Now, we turn our attention to the relation between bilinear \cz operators
and bilinear pseudodifferential operators.
A bilinear pseudodifferential operator $T_\sigma$ with a symbol $\s$,
a priori defined from $\mathcal{S} \times \mathcal{S}$
into $\mathcal{S}'$, is given by
%
\begin{equation}
T_\sigma(f,g)(x)=\int_{\rn}\int_{\rn} \sigma (x, \xi, \eta )\widehat f(\xi )\widehat g(\eta )
e^{ix\cdot(\xi + \eta )}d\xi d\eta.
\label{T2}
\end{equation}

\noi
We say that a symbol $\s$ belongs the
bilinear class $BS^m_{\rho, \dl}$
if 
\begin{equation}
\label{bi-pseudo} |\partial _x^\alpha\partial _\xi ^\beta\partial _\eta
^\gamma \sigma (x, \xi ,\eta )|\lesssim \big(1+|\xi
|+|\eta |\big) ^{m+\delta |\alpha |-\rho (|\beta |+|\gamma |)}
\end{equation}
for all \((x, \xi, \eta )\in {\re}^{3n}\) and all multi-indices
\(\alpha,\) \(\beta\) and \(\gamma\).
Such symbols are commonly referred to as
\emph{bilinear H\"ormander pseudodifferential symbols}. The collection of bilinear pseudodifferential operators with symbols in \(BS_{\rho, \delta}^m\) will be denoted by \(\mathcal{O}p \, BS^m_{\rho, \delta}\). Note that, for example,
 operators in \(\mathcal{O}p \, BS^m_{\rho, \delta}\) model  the product of two functions and their derivatives.

It is a known fact that bilinear \cz kernels correspond to bilinear pseudodifferential symbols in the class \(BS_{1, 1}^0\), see \cite{GT}. Moreover, \cz operators are ``essentially the same'' as pseudodifferential operators with symbols in the subclass \(BS_{1, \delta}^0\), \(0\leq\delta<1\), a fact that in turn is tightly connected to the existence of a symbolic calculus for \(BS_{1, \delta}^0\), see \cite{BMNT}.

\begin{theoremd*}
Let $\sigma\in BS_{1, \delta}^0$, $0\leq\delta<1$. Then,  $T_{\sigma}^{*j}=T_{\sigma^{*j}}$ with $\sigma^{*j}\in BS_{1, \delta}^0$, $j=1, 2$, and $T_\sigma$ is a bilinear Calder\'on-Zygmund operator.
\end{theoremd*}

\noi
Thus, we can view bilinear \cz operators on the frequency side as operators given by
\eqref{T2}
with symbols \(\sigma \in BS^m_{\rho, \dl}\), where \(\rho=1, 0\leq \delta<1\) and \(m=0\).

Our main interest is to consider  the previously defined bilinear operators under the additional operation of commutation. For a bilinear operator \(T\), and  (multiplicative) functions $b, b_1,$ and $ b_2$ , we consider the following three \emph{bilinear commutators}:
\begin{align*}
[T, b]_1(f, g)&= T(bf,g)-bT(f,g),\\
[T, b]_2 (f, g)&=T(f,bg)-bT(f,g),\\
[[T, b_1]_{_1}, b_2]_2 (f, g)&= [T, b_1]_{_1}(f,b_2g) - b_2[T, b_1]_{_1}(f,g).
\end{align*}

\noi
First, we consider the case when  \(T\) is a bilinear \cz operator with kernel \(K\) and
 \(b, b_1, b_2\) belong to \(BMO(\rn)\). Then, the three bilinear commutators can formally be written as
\begin{align*}\label{formally}
[T, b]_1(f, g)(x)&=\int_{\rn}\int_{\rn} K(x, y, z)\big(b(y)-b(x)\big)f(y)g(z)\, dydz,\\
[T, b]_2 (f, g)(x)&=\int_{\rn}\int_{\rn} K(x, y, z)\big(b(z)-b(x)\big)f(y)g(z)\, dydz,\\
[[T, b_1]_{_1}, b_2]_2 (f, g)(x)&=\int_{\rn}\int_{\rn}  K(x, y, z)\big(b_1(y)-b_1(x)\big)
\big(b_2(z)-b_2(x)\big)f(y)g(z)\, dydz.
\end{align*}
As in the linear case, these operators are bounded from \(L^p\times L^q\rightarrow L^r\) with \(1/p+1/q=1/r\) for all \(1<p,q<\infty\),  see Grafakos and Torres \cite{GT1}, Perez and Torres \cite{PT}, Perez et
al.~\cite{PPTT} and Tang \cite{Tan}, with estimates of the form
\begin{align*}
&
\big\|[T, b]_1(f, g)\big\|_{L^r},  \|[T, b]_2(f, g)\|_{L^r}\lesssim \|b\|_{BMO}\|f\|_{L^p}\|g\|_{L^q},\\
& \vphantom{\Big|}
\big\|[[T, b_1]_{_1}, b_2]_2 (f, g)\big\|_{L^r}\lesssim \|b_1\|_{BMO}\|b_2\|_{BMO}\|f\|_{L^p}\|g\|_{L^q}.
\end{align*}
However, the bilinear commutators obey a ``smoothing effect" and are, in fact, even better behaved if we allow the symbols $b$ to be slightly smoother. The following theorem of B\'enyi and Torres \cite{BT1}, should be regarded as the bilinear counterpart of the result of Uchiyama \cite{Uch}
mentioned before.

\begin{theoreme*}
Let \(T\) be a bilinear \cz operator. If \(b\in CMO\),  \(1/p+1/q=1/r\), \(1<p,q<\infty\) and  \(1\leq r <\infty\), then \([T, b]_1: L^p\times L^q\to L^r\) is a bilinear compact operator. Similarly,  if \(b_1, b_2\in CMO\), then \([T,b_2]_2\)  and \([[T, b_1]_{_1}, b_2]_2\) are bilinear compact operators for the same range of exponents.
\end{theoreme*}

\noi
Interestingly, the notion of compactness in the multilinear setting alluded to in Theorem~E can be traced back to the foundational article of Calder\'on \cite{Cal3}. Given three normed spaces \(X, Y, Z\), a bilinear operator \(T: X\times Y\to Z\) is called \emph{(jointly) compact} if the set \(\{T(x, y): \|x\|, \|y\|\leq 1\}\) is precompact in \(Z\). Clearly, any compact bilinear operator \(T\) is continuous;
 for further connections between this and other notions of compactness, see again \cite{BT1}. An immediate consequence of Theorems~D and~E is the following compactness result for commutators of bilinear pseudodifferential operators.

\begin{corollaryf*}
Let \(\sigma\in BS_{1, \delta}^0\), \(0\leq\delta<1\), and \(b, b_1, b_2\in CMO\).
Then, \([T_\sigma, b]_i, i=1,2,\) and \([[T_\sigma, b_1]_{_1}, b_2]_2\) are bilinear compact operators from \(L^p\times L^q\to L^r\) for \(1/p+1/q=1/r, 1<p,q<\infty\) and  \(1\leq r <\infty\).
\end{corollaryf*}

Varying the parameters \(\rho, \delta\) and \( m\) in the definition of the bilinear H\"ormander classes \(BS_{\rho, \delta}^m\) 
is a way of escaping the realm of bilinear Calder\'on-Zygmund theory. In this context, it is useful to recall the following statement from \cite{BBMNT}.

\begin{theoremg*}
Let \(0\leq \delta\leq \rho\leq 1,\) \(\delta<1,\) \(1\leq p, q\leq
\infty,\) \(0<r<\infty\) be such that \(1/p+1/q=1/r\),
\[
m<m(p, q):=n(\rho -1)\left(\max\Big(\frac{1}{2},\,\frac{1}{p},\, \frac{1}{q},\, 1-\frac{1}{r}\Big)+\max\Big(\frac{1}{r}-1,0\Big)\right),
\]
and \(\sigma\in BS^m_{\rho,\delta}(\rn).\)
Then,  \(T_\sigma\) extends to a bounded operator from  \(L^p\times L^q\to L^r\).
\end{theoremg*}

\noi
 See also Miyachi and Tomita \cite{MT} for the optimality of the order $m$ and the extension of the result in \cite{BBMNT} below $r=1$.

Clearly, the class \(BS_{1, 0}^1\) falls outside the scope of Theorem~F; since \(\rho=1\), the only way to make the class \(BS_{1, \delta}^m\), \(0\leq \delta<1\), to produce operators that are bounded is to require the order \(m < 0\). However, guided by the experience we gained in the linear case, it is natural to hope that the phenomenon of smoothing of bilinear commutators manifests itself again in the bilinear context of pseudodifferential operators. This is confirmed by our main results, Theorem~\ref{main1} and Theorem~\ref{main2}, which we now state.

\begin{theorem}
\label{main1}
Let  \(T_\sigma\in \mathcal{O}p \, BS^1_{1, 0}\) and $a$ be a Lipschitz function such that \(\nabla a\in L^\infty\).
Then,  \([T_\sigma, a]_i, i=1, 2,\) are bilinear \cz operators. In particular,  \([T_\sigma, a]_i, i=1, 2,\) are bounded  from \(L^p\times L^q\to L^r\) for \(1/p+1/q=1/r, 1<p,q<\infty\) and  \(1\leq r <\infty\).
\end{theorem}

\noi
Once we prove that the commutators
\([T_\sigma, a]_i, i=1, 2,\)
are bilinear \cz operators,  
 the end-point boundedness results directly follow from Theorem~B. Theorem~\ref{main1} also admits a natural converse, see the remark at the end of this paper; thus making Theorem~\ref{main1} the natural bilinear extension of Theorem~A.

Combining Theorem~\ref{main1} with Theorem~E, we immediately obtain the following compactness result for the iteration of commutators.

\begin{theorem}
\label{main2}
Let \(T_\sigma\in \mathcal{O}p \, BS^1_{1, 0}\),
$a$  be a Lipschitz function such that \(\nabla a\in L^\infty\), and $b, b_1, b_2\in CMO$.
Then,  \([[T_\sigma, a]_i, b]_j\), \(i, j=1, 2,\) and \([[[T_\sigma, a]_i, b_1]_1, b_2]_2\), \(i=1, 2,\) are bilinear compact operators from \(L^p\times L^q\to L^r\) for \(1/p+1/q=1/r, 1<p,q<\infty\) and  \(1\leq r <\infty\).
\end{theorem}

The remainder of our paper is devoted to the proof of Theorem~\ref{main1}. While the argument we present is influenced by Coifman and Meyer's exposition of the linear case, see \cite[Theorem 4, Chapter 9]{MC}, there are several technical obstacles in the bilinear setting that must be overcome.

\section{Proof of Theorem 1}\label{sec:2}

The proof can be summarized in the following statement:
the kernels of the commutators are indeed bilinear \cz
and
the commutators verify the conditions (i) and (ii)
in the $T(1)$ theorem (Theorem~C) from the bilinear Calder\'on-Zygmund theory.

We divide the proof of Theorem~\ref{main1} into several subsections. In Subsection~\ref{subsec:1},  we show that the kernels of the commutators $[T_\sigma, a]_i, i=1, 2,$ are Calder\'on-Zygmund. Sections~\ref{subsec:2}-\ref{subsec:4} are devoted to proving that the commutators satisfy the cancelation condition (ii)  in Theorem~C. Finally, in Subsection~\ref{subsec:5},  we prove that the commutators verify the bilinear weak boundedness property.

In the following,
 $a$ denotes a Lipschitz function such that $\nabla a\in L^\infty$ and
  $T=T_\sigma$ is  the bilinear pseudodifferential operator associated to
a symbol $\s\in BS_{1, 0}^1$, that is,~$\s$ satisfies
\begin{equation}
\label{bs110}
|\dd^\al_x \dd^\be_\xi \dd_\eta^\g \s(x, \xi, \eta) | \les \big(1 + |\xi| + |\eta|\big)^{1 - |\be|- |\g|},
\end{equation}

\noi
for all $x, \xi, \eta\in\re^n$ and all multi-indices $\al, \be, \g$.

\subsection
{Bilinear  \cz kernels}\label{subsec:1}

Let $K_j$  be the kernel of $[T, a]_j$, $j = 1, 2$. Then, we have
\begin{align*}
K_1(x, y, z) & = \big(a(y) - a(x)\big) K(x, y, z),\\
K_2(x, y, z) & = \big(a(z) - a(x)\big) K(x, y, z),
\end{align*}
where $K$ is the kernel of $T$.
Note that $K$ can be written (up to a multiplicative constant) as
\begin{equation}
 K(x, y, z) =  \iint e^{i\xi \cdot (x - y)}  e^{i\eta \cdot (x - z)} \s(x, \xi, \eta) d \xi d \eta.
\label{Q1}
 \end{equation}

\noi
There are certain decay estimates on $\dd^\al_x \dd^\be_y \dd^\g_z K(x, y, z)$,
when $x \ne y$ or $x \ne z$.

\begin{lemma}
\label{L00}
The kernel $K$ satisfies
\[
|\dd_x^\al \dd_y^\be \dd_z^\g K (x, y, z)|
\leq C(\al, \be, \g) \big(|x - y|+ |x-z|\big)^{-2n-1 -|\al|-|\be| - |\g|}.
\]

\noi
when $x \ne y$ or $x \ne z$.
\end{lemma}

\noi
Assuming  Lemma~\ref{L00}, we can show the desired result about the kernels $K_1$ and $K_2$.

\begin{lemma}
\label{L0}
$K_1$ and $K_2$ are bilinear Calder\'on-Zygmund kernels.
\end{lemma}

\begin{proof}
By Lemma~\ref{L00} and noting that $|x-y| + |x- z|+ |y -z|\sim |x-y|+|x-z|$, we have
\begin{align*}
|K_1(x, y, z) |, |K_2(x, y, z) | & \les \|\nabla a\|_{L^\infty} \big(|x-y| + |x- z|+ |y -z|\big)^{-2n}, \\
|\nabla K_1(x, y, z) |,
|\nabla K_2(x, y, z) | & \les \|\nabla a\|_{L^\infty} \big(|x-y| + |x-z| + |y -z|\big)^{-2n - 1},
\end{align*}

\noi
on $\mathbb R^{3n}\setminus \Omega$, where
\(\Omega= \{(x,y,z) \in \re^{3n}: x=y=z \}\).
\end{proof}

The remainder of this subsection is devoted to the proof of Lemma~\ref{L00}.

\begin{proof}[Proof of Lemma \ref{L00}]
Let $\psi$ be a smooth cutoff function supported on $\{\xi\in\re^n: |\xi| \leq 2\}$
such that $\psi(\xi) = 1$ for $|\xi| \leq 1$.
For $N \in \mathbb{N}$, let $\psi_N(\xi, \eta) = \psi(\frac{\xi}{N})\psi(\frac{\eta}{N})$.
Note that
\begin{equation}
|\dd^\be_\xi \dd_\eta^\g \psi_N(\xi, \eta)| =
\begin{cases}
O( N^{-|\be|-|\g|}) = O \big((|\xi|+|\eta|)^{-|\be| - |\g|}\big), & N \leq |\xi|, |\eta| \leq 2N,\\
0, & \text{otherwise}.
\end{cases}
\label{Q2}
\end{equation}

\noi
for $(\be, \g) \ne (0, 0)$.
Moreover, for $\be \ne 0$, we have
\begin{equation}
|\dd^\be_\xi \psi_N(\xi, \eta) |=
\begin{cases}
O( N^{-|\be|}) = O \big((|\xi|+|\eta|)^{-|\be| - |\g|}\big), & N \leq |\xi|  \leq 2N,\\
0, & \text{otherwise}.
\end{cases}
\label{Q2a}
\end{equation}

\noi
since $\psi_N$ is non-trivial only if $|\eta| \leq 2N$.
A similar estimate holds for $|\dd_\eta^\g \psi_N (\xi, \eta)|$, $\g \ne 0$.
Hence, we have
\begin{equation}
\s_N (x, \xi, \eta) := \s(x, \xi, \eta) \psi_N (\xi, \eta) \in B S^1_{1, 0}
\label{Q3}
\end{equation}

\noi
and, moreover,
we have
$|\dd^\al_x \dd^\be_\xi \dd_\eta^\g \s_N(x, \xi, \eta) | \les (1 + |\xi| + |\eta|)^{1 - |\be| - |\g|}$,
where the implicit constant is independent of $N$.
Now, let
\[ K_N(x, y, z) =  \iint
e^{i\xi \cdot (x - y)}  e^{i\eta \cdot (x - z)} \s_N (x, \xi, \eta) d \xi d \eta.\]

\noi
In the following, we show that
\begin{equation}
| \dd_x^\al \dd_y^\be \dd_z^\g K_N(x, y, z)|
\leq C(\al, \be, \g) \big(|x - y|+ |x-z|\big)^{-2n -1 -|\al|-|\be| - |\g|}
\label{Q4}
\end{equation}

\noi
uniformly in $N$. Since $\s_N (x, \xi, \eta)$ converges 
pointwise to $\s (x, \xi, \eta)$, it follows that $K_N$ converges to $K$ in the sense of distributions. This in turn shows that the estimates in \eqref{Q4} hold for $K(x, y, z)$ as well, yielding our lemma. The remainder of the proof is therefore concerned with \eqref{Q4}.

\smallskip

First, we consider the case $\al =\be = \g = 0$, that is, we estimate $K_N(x, y, z)$. Without loss of generality, let us assume that $|x - y| \geq |x-z|$; in particular, we have $|x- y| \sim |x- y| + |x-z|$.

\smallskip

\noi {\bf Case (i):} $|x - y| \geq 1$.

Note that
\( e^{i\xi \cdot (x - y) } = - \displaystyle\frac{1}{|x-y|^2} \Dl_\xi e^{i\xi \cdot (x - y) }.\) Let $m\in\mathbb N$ be such that $2m-1>2n$. Then, integrating by parts, we have
\begin{align*}
| K_N(x, y, z)  |
& = \frac{1}{|x-y|^{2m}}
\bigg|\iint e^{i\xi \cdot (x - y)} e^{i\eta \cdot (x - y)}
\Dl_\xi^m \s_N(x, \xi, \eta) d \xi d \eta\bigg|\notag \\
& \les
  \frac{1}{|x-y|^{2m}}
\iint
\frac{1}{(1+|\xi|+|\eta|)^{2m-1} }
d \xi d \eta\notag \\
& \leq
  \frac{1}{|x-y|^{2m}}
\int \frac{1}{(1+|\xi|)^{m - \frac 12} }
d \xi
\int
\frac{1}{(1+|\eta|)^{m-\frac 12} }
 d \eta\notag \\
& \les |x - y|^{-2m}
\leq |x - y|^{-2n - 1}.
\end{align*}
Hence, \eqref{Q4} holds in this case.

\medskip

\noi
{\bf Case (ii):} $|x - y| < 1$.

Fix $x, y$ with $x \ne y$
and let $ r= |x - y| \sim|x-y| + |x - z|$.
Then, write $x- y $
as
\[x - y = r u \]

\noi
for some unit vector $u$.
With the smooth cutoff function  $\psi$ supported on $\{\xi\in\re^n: |\xi|\leq 2\}$ as above,
 define $\wt \psi = 1 - \psi$. Then, by a change of variables, we have
\begin{align}
 K_N(x, y, z)
 & = \frac{1}{r^{2n}} \iint e^{i\xi \cdot u} e^{i r^{-1}\eta\cdot (x - z)}
  \s_N(x, r^{-1} \xi, r^{-1}\eta ) d \xi d\eta \notag\\
 & = \frac{1}{r^{2n}} \iint e^{i\xi \cdot u} e^{i r^{-1}\eta\cdot (x - z)}
  \s_N(x, r^{-1} \xi, r^{-1}\eta ) \psi(\eta)d \xi d\eta \notag\\
& \hphantom{XXXX}
+ \frac{1}{r^{2n}}
 \iint e^{i\xi \cdot u} e^{i r^{-1}\eta\cdot (x - z)}
  \s_N(x,r^{-1} \xi, r^{-1}\eta ) \wt  \psi(\eta)d \xi d\eta \notag\\
& = :  K^0_N(x, y, z)  +  K^1_N(x, y, z).
\label{Q5}
 \end{align}
Then,  by inserting another cutoff in $\xi$, we write $K^0_N$ as
\begin{align}
 K_N^0(x, y, z)
 & = \frac{1}{r^{2n}} \iint e^{i\xi \cdot u} e^{i r^{-1}\eta\cdot (x - z)}
  \s_N(x, r^{-1} \xi, r^{-1}\eta ) \psi(\xi) \psi(\eta)d \xi d\eta \notag\\
& \hphantom{XXXX}
+ \frac{1}{r^{2n}}
 \iint e^{i\xi \cdot u} e^{i r^{-1}\eta\cdot (x - z)}
  \s_N(x, r^{-1} \xi, r^{-1}\eta ) \wt \psi(\xi) \psi(\eta)d \xi d\eta \notag\\
& = :  K^2_N(x, y, z)  +  K^3_N(x, y, z) .
\label{Q6}
 \end{align}

We begin by estimating $K^2_N$.
Since $|\s_N(x, r^{-1} \xi, r^{-1} \eta) | \les r^{-1}$ on $\{ |\xi|, |\eta| \leq 2\}$,
we have
\begin{align}
| K^2_N(x, y, z) | \les r^{-2n -1} \sim \big( |x-y|+ |x - z|\big)^{-2n-1}.
\label{Q7}
\end{align}


Note now that
\begin{align}
	 |\dd_\xi^\be  \dd_\eta^\g
\s_N(x, r^{-1} \xi, r^{-1} \eta) |
& = r^{-|\be|-|\g|}|\dd_2^\be \dd_3^\g
\s_N(x, r^{-1} \xi, r^{-1}\eta) | \notag \\
&  \les r^{-1} (r + |\xi| + |\eta|)^{1-|\be|-|\g|}\notag\\
& \les r^{-1} (1 + |\xi| + |\eta|)^{1-|\be|-|\g|},
\label{Q8}
\end{align}

\noi
where the last inequality holds
if  $|\xi|\geq 1$ or $|\eta|\geq 1$.
%
Then, proceeding as before with integration by parts and using  \eqref{Q8}, we have
\begin{align}
|K_N^1(x, y, z) |
& = \frac{1}{r^{2n}}
\bigg|\iint e^{i\xi \cdot u} e^{i r^{-1} \eta \cdot (x - z)}
\Dl_\xi^m \s_N(x, r^{-1} \xi, r^{-1}\eta) \wt \psi(\eta) d \xi d \eta \bigg|\notag \\
& \les
r^{-2n-1}
\iint
\frac{1}{(1+|\xi|+|\eta|)^{2m-1} }
d \xi d \eta\notag \\
%
& \les r^{-2n-1},
\label{Q9}
\end{align}

\noi
as long as  $2m - 1 > 2n$. Similarly, integrating by parts with \eqref{Q8} and noting that, for $\be \ne 0$, we have
$ \dd_\xi^\beta \wt \psi(\xi) = 0$ unless $|\xi| \in [1, 2]$,
we have
\begin{align}
|K_N^3(x, y, z)  |
& = \frac{1}{r^{2n}}
\bigg| \iint e^{i\xi \cdot u} e^{i r^{-1}\eta\cdot (x - z)}
\Dl_\xi^m \big(  \s_N(x, r^{-1} \xi, r^{-1}\eta ) \wt \psi(\xi)\big) \psi(\eta)d \xi d\eta\bigg| \notag\\
& \les
r^{-2n-1}
+  \frac{1}{r^{2n}}
\bigg|\iintt_{|\xi|\geq 1, |\eta|\leq 2} e^{i\xi \cdot u} e^{i r^{-1}\eta\cdot (x - z)}
\Dl_\xi^m \big(  \s_N(x, r^{-1} \xi, r^{-1}\eta ) \big)
\wt \psi(\xi) \psi(\eta) d \xi\bigg|\notag \\
& \les r^{-2n-1},
\label{Q10}
\end{align}

\noi
as long as  $2m - 1 > n$ in this case. Finally, combining the estimates \eqref{Q7},  \eqref{Q9}, and \eqref{Q10} yields \eqref{Q4}.

\smallskip

Next, we consider the case $(\al, \be, \g) \ne (0, 0, 0)$.
Note that $\xi^{\wt \be} \eta^{\wt \g}\dd_x^\theta \s_N \in BS^{1+|\wt \be| + |\wt \g|}_{1, 0}$,
where the implicit constant on the bounds
of the derivatives of
$\xi^{\wt \be} \eta^{\wt \g}  \dd_x^\theta \s_N $ is independent of $N$ and $\theta$.
Then, we have
\[ \dd_x^\al \dd_y^\be \dd_z^\g
K_N(x, y, z) =  \iint e^{i\xi \cdot (x - y)} e^{i \eta\cdot(x - z)} \wt \s_N(x, \xi, \eta)  d \xi d \eta,\]

\noi
for some
$\wt \s_N \in BS^{1+ |\al|+|\be|+|\g|}_{1, 0}$.

When $|x - y|\geq 1$, we can repeat the computation in Case (i)
and obtain \eqref{Q4} by choosing $2 m - 1 - |\al| - |\be| - |\g| > 2n$.
Now, assume $|x - y| < 1$. For $K^2_N$,  it suffices to note that
$|\wt \s_N(x, r^{-1} \xi, r^{-1} \eta) | \les r^{-1 - |\al| - |\be| - |\g|}$ on $\{ |\xi|, |\eta| \leq 2\}$.
For $K^1_N$ and $K^3_N$, we note that
\begin{align*}
	 |\dd_\xi^{\wt \be}  \dd_\eta^{\wt \g}
\wt \s_N(x, r^{-1} \xi, r^{-1} \eta) |
& = r^{-|\wt \be|-|\wt \g|}|\dd_2^{\wt \be} \dd_3^{\wt \g}
\wt \s_N(x, r^{-1} \xi, r^{-1}\eta) | \notag \\
&  \les r^{-1-|\al| - |\be|-|\g|} (r + |\xi| + |\eta|)^{1+ |\al|+|\be|+|\g|-|\wt \be|-|\wt \g|}\notag\\
&  \les r^{-1-|\al| - |\be|-|\g|} (1 + |\xi| + |\eta|)^{1+ |\al|+|\be|+|\g|-|\wt \be|-|\wt \g|},
\end{align*}

\noi
where the last inequality holds if  $|\xi|\geq 1$ or $|\eta|\geq 1$.
The rest follows as in Case (ii).
\end{proof}

\subsection{A representation of the class $BS_{1, 0}^1$ via $BS_{1, 0}^0$}\label{subsec:2}

Without loss of generality, we will assume that $\s( x, 0, 0)=0$. This is possible because
even if we  replace $\s$ by $\s_0$, where $\s_0(x, \xi , \eta) = \s(x, \xi , \eta) - \s(x, 0, 0),$
the commutators are unchanged. Namely, $[T_\s, a]_j = [T_{\s_0}, a]_j$ for $j = 1, 2$.
Note that  $\s_0( x, 0, 0)=0$ and $\s_0\in BS_{1, 0}^1$.
 We can further assume that $\s$ has compact support; this justifies the manipulations in the following.
 A standard limiting argument then removes this additional assumption; see, for example, the discussion about loosely convergent sequences of $BS_{\rho, \delta}^m$ symbols in \cite{BT}, also Stein \cite[pp. 232-233]{S}.

\begin{lemma}
\label{L1}
The symbol $\s \in BS^1_{1, 0}$ has the representation
$\s = \sum_{j = 1}^n (\xi_j \s_j  + \eta_j \wt{\s}_j),$
where $\s_j, \wt{\s}_j \in BS^0_{1, 0}$.
In particular, if $T_j$ and $\wt T_j$ are the bilinear pseudodifferential operators corresponding to $\s_j$ and $\wt \s_j$, respectively, then we have
\[
T(f, g) = \sum_{j = 1}^n \big[T_j (D_j f, g) + \wt T_j (f, D_jg)\big],
\]

\noi
where $T = T_\s \in  \mathcal{O}p BS^1_{1, 0}$.
\end{lemma}

\begin{proof}
By the Fundamental Theorem of Calculus
with $\zeta = (\xi, \eta)$, we have
\begin{align*}
 \s(x, \xi, \eta)
 & =  \s(x, \xi, \eta) - \s(x, 0, 0)
 = \zeta \cdot \int_0^1 \nabla_{ \zeta'}  \s(x,\zeta' )\Big|_{\zeta' = t \zeta} dt\\
& = \sum_{j = 1}^n \big[\xi_j \s_j(x, \xi, \eta)
+ \eta_j \wt \s_j(x, \xi, \eta) \big],
 \end{align*}
where the symbols  $\s_j$ and $\wt \s_j$ are given by
\[\s_j(x, \xi, \eta )  = \int_0^1 \dd_{\xi'_j} \s(x, \xi' , t \eta)
\Big|_{\xi' = t \xi} dt \quad \text{and} \quad
\wt \s_j(x, \xi, \eta)  =
\int_0^1 \dd_{\eta'_j} \s(x, t\xi, \eta' )
\Big|_{\eta' = t \eta} dt.\]

It remains to show that $\s_j, \wt \s_j \in BS^0_{1, 0}$. First, note that, for $t \in [0, 1]$, we have
\begin{align}
t \big(1+t( |\xi|+ |\eta|) \big)^{-1}  \les (1+ |\xi| + |\eta| )^{-1} .
\label{XX}
\end{align}

\noi
By exchanging the differentiation with integration and applying \eqref{XX}, we have
\begin{align*}
|\dd_x^\al \dd_\xi^\be \dd_\eta^\g \s_j(x, \xi, \eta)|
& = \bigg| \int_0^1
 t^{|\be|+ |\g|}
\dd_x^\al \dd_{\xi'}^\be \dd_{\eta'}^\g \dd_{\xi'_j} \s(x, \xi', \eta' )\Big|_{(\xi', \eta') = t(\xi, \eta)}
 dt\bigg|\\
& \les
\int_0^1  t^{|\be|+ |\g|} \big(1+t (|\xi|+|\eta|)\big)^{-(|\be|+|\g|)}  dt\\
& \les (1+|\xi|+ |\eta|)^{-(|\be|+|\g|)},
\end{align*}

\noi
Therefore, $\s_j \in BS^0_{1, 0}$.
A similar argument shows that
$\wt \s_j \in BS^0_{1, 0}$.
\end{proof}

\subsection{Transposes of bilinear commutators}\label{subsec:3}

Recall that the commutators $[T, a]_1$
and $[T, a]_2$ are defined as
\begin{align}
[T, a]_1(f, g) & = T(af, g) - aT(f, g), \label{Z1}\\
[T, a]_2(f, g) & = T(f, ag) - aT(f, g).\label{Z2}
\end{align}
Given a bilinear operator $T$, the transposes $T^{*1}$ and $T^{*2}$ are defined by
\begin{equation*}
\jb{T(f, g), h}
= \jb{T^{*1}(h, g), f}
= \jb{T^{*2}(f, h), g},
\end{equation*}
where $\jb{\cdot, \cdot}$ denotes the dual pairing.

\begin{lemma}
\label{L2}
We have the following identities:
\begin{align}
\big([T, a]_1\big)^{*1}& =  - [T^{*1}, a]_1, \label{Z5}\\
\big([T, a]_1\big)^{*2}& =   [T^{*2}, a]_1-[T^{*2}, a]_2. \label{Z6}
\end{align}
Similarly, we have
\begin{align}
\big([T, a]_2\big)^{*1}& =  [T^{*1}, a]_2-[T^{*1}, a]_1,
\label{Z7}\\
\big([T, a]_2\big)^{*2}& =   -[T^{*2}, a]_2. \label{Z8}
\end{align}
\end{lemma}

\begin{proof}
We briefly indicate the calculations that give ~\eqref{Z5} and ~\eqref{Z6}. The following sequence of equalities yields ~\eqref{Z5}:
\begin{align*}
\jb{[T, a]_1(f, g), h}
& = \jb{T(af, g), h}
 - \jb{aT(f, g), h}
= \jb{T^{*1}(h, g), af}
 - \jb{T(f, g), a h}\\
& = \jb{aT^{*1}(h, g), f}
 - \jb{T^{*1} (ah, g), f}
 = \jb{ - [T^{*1}, a]_1(h, g), f}.
\end{align*}
We also have
\begin{align*}
\jb{[T, a]_1(f, g), h}
& = \jb{T^{*2}(af, h), g}
 - \jb{T^{*2} (f, ah), g}\\
& = \jb{T^{*2}(af, h), g}
- \jb{a T^{*2}(f, h), g}
 - \big( \jb{T^{*2} (f, ah), g}
  - \jb{a T^{*2}(f, h), g}\big)
 \\
&
 = \jb{ [T^{*2}, a]_1(f, h), g}
 -
\jb{ [T^{*2}, a]_2(f, h), g},
\end{align*}
thus proving ~\eqref{Z6}.
The identities \eqref{Z7} and \eqref{Z8} follow in a similar manner.
\end{proof}

\subsection{Cancelation conditions for bilinear commutators}\label{subsec:4}

We will prove here that the commutators satisfy the $BMO$ bounds in the bilinear $T(1)$ theorem (Theorem~C).

\begin{lemma}
\label{L3}
Let $T \in \mathcal{O}p BS^1_{1, 0}$
and $a$ be a Lipschitz function.
Then, we have $[T, a]_j\in BMO, j=1,2$.
\end{lemma}

\begin{proof}
By Lemma~\ref{L1}, we have
\begin{align*}
[T, a]_1 (1, 1)& = T(a, 1) - \underbrace{aT(1, 1)}_{= 0}
= \sum_{j = 1}^n \big[ T_j (D_j a, 1) + \underbrace{\wt T_j (a, D_j 1)}_{ =0} \big] \\
& = \sum_{j = 1}^n  T_j (D_j a, 1).
\end{align*}

\noi
It follows from Theorem~D that $T_j \in \mathcal{O}p\,BS^0_{1, 0}$ are bilinear Calder\'on-Zygmund operators.
Then, by Theorem~B, we obtain that $T_j (D_j a, 1)\in BMO$, since $D_j a \in L^\infty$.
Therefore, we conclude that  $[T, a]_1 (1, 1)\in BMO$.


Similarly, we have
\begin{align*}
[T, a]_2(1, 1) & = T(1, a) - \underbrace{aT(1, 1)}_{= 0}
 = \sum_{j = 1}^n \big[ \underbrace{T_j (D_j 1, a)}_{= 0} + \wt T_j (1, D_j a) \big] \\
& = \sum_{j = 1}^n  \wt T_j (1, D_j a)\in BMO,
\end{align*}
since $D_j a \in L^\infty$ and $\wt T_j \in \mathcal{O}p\,BS^0_{1, 0}$.
\end{proof}

\begin{lemma}
\label{L4}
Let $T$ and $a$ be as in Lemma \ref{L3}.
Then, we have  $[T, a]_j^{*i}\in BMO, \,  i, j=1,2$.
\end{lemma}

\begin{proof}
From Theorem 2.1 in \cite{BMNT},  we know that if $T \in \mathcal{O}p\,BS^1_{1, 0}$, then $T^{*1}, T^{*2} \in \mathcal{O}p\,BS^1_{1, 0}$ as well. By Lemma~\ref{L2}, for $i=1, 2$, the transposes $[T, a]_1^{*i}$ and $[T, a]_2^{*i}$ consist of commutators of $T^{*1}$ and $T^{*2}$ with the Lipschitz function $a$. The conclusion now follows from Lemma~\ref{L3}.
\end{proof}

\subsection{The weak boundedness property for bilinear commutators}\label{subsec:5}

A function $\phi \in \mathcal{D}$
is called a \emph{normalized bump function of order $M$} if
$\supp \phi \subset B_0(1)$
and $\|\dd^\al \phi\|_{L^\infty} \leq 1$ for all multi-indices $\al$ with $|\al|\leq M$.
Here, $B_x (r)$ denotes the ball of radius $r$ centered at $x$.

We say that a bilinear singular integral operator
$T:\mathcal S \times \mathcal S\to \mathcal S'$
has the \emph{(bilinear) weak boundedness property}
if there exists $M \in \mathbb{N}\cup\{0\}$
such that for all normalized bump functions $\phi_1, \phi_2$, and $\phi_3$
of order $M$, $x_1, x_2, x_3 \in \re^n$ and $t > 0$,
we have
\begin{equation}
\big|\jb{T(\phi_1^{x_1, t}, \phi_2^{x_2, t}), \phi_3^{x_3, t})}\big|
\les t^n,
\label{Z9}
\end{equation}
	
\noi
where $\phi_j^{x_j, t}(x) = \phi_j\big(\frac{x - x_j}{t}\big)$.
Note that
\begin{equation}
\| \partial_x^\al  \phi_j^{x_j, t}\|_{L^p} \les t^{\frac{n}{p} - |\al|}.
\label{Z10}
\end{equation}

\noi
The following lemma provides a simplification of the condition \eqref{Z9}.

\begin{lemma} \label{L8}
Let $T$ be a bilinear operator defined by \eqref{T1}
with a bilinear \cz kernel $K$, satisfying \eqref{cz-size}.
Then,  the weak boundedness property holds
if  there exists $M \in \mathbb{N}\cup\{0\}$ such that
\begin{equation}
\big|\jb{T(\phi_1^{x_0, t}, \phi_2^{x_0, t}), \phi_3^{x_0, t})}\big|
\les t^n,
\label{X1}
\end{equation}

\noi
for
 all normalized bump functions $\phi_1, \phi_2$, and $\phi_3$
of order $M$, $x_0 \in \re^n$ and $t > 0$.

\end{lemma}

\begin{proof}
Suppose that $T$ satisfies \eqref{X1} for some fixed $M$.
Fix $t > 0$
and
normalized bump functions $\phi_1, \phi_2$ and $\phi_3$ of order $M$
in the following.

\medskip

\noi
{\bf Case (i)}
Suppose that $|x_1 - x_3 |, |x_2 - x_3 | \leq  3t$.
For $j = 1, 2$, we define $ \psi_j$ by setting
\begin{equation*}
 \psi_j^{x_3, 4t}(x) =  \psi_j\big(\tfrac{x - x_3}{4t}\big)
:= \begin{cases}
4^{-M} \phi_j^{x_j, t}(x), & \text{if } x \in B_{x_j}(t), \\
0, & \text{otherwise}.
\end{cases}
\end{equation*}

\noi
Note that $ \psi_j$ is
a normalized bump function of order $M$.
For $j = 3$,  let $ \psi_3(x) = 4^{-M} \phi_3 (4x)$.
Note that $ \psi_3$ is also a normalized bump function of order $M$.
Then, by \eqref{X1}, we have
\begin{align*}
\big|\jb{T(\phi_1^{x_1, t}, \phi_2^{x_2, t}), \phi_3^{x_3, t})}\big|
= 4^{3M}
\big|\jb{T( \psi_1^{ x_3, 4t},  \psi_2^{x_3, 4t}),  \psi_3^{x_3,4t})}\big|
\les 4^{3M + n} t^n \sim t^n.
\end{align*}

\medskip

\noi
{\bf Case (ii)}
Suppose that $\max\big(|x_1 - x_3 |, |x_2 - x_3 |\big) >  3t$.
For the sake of the argument, 
suppose that $|x_1-x_3|>3t$. 
Then, by the triangle inequality, we have  
$|x-y|>|x_1-x_3|-|x-x_3|-|y-x_1|>t$
for 
for all $ x\in B_{x_3}(t)$ and $ y\in B_{x_1}(t)$.
A similar calculation shows that if $|x_2-x_3|>3t$, 
then we have $|x - z| > t$ 
for all $ x\in B_{x_3}(t)$ and $ z\in B_{x_2}(t)$.
Hence, we have
%
\[\max \big(|x-y|, |x-z|\big) > t\]

\noi
for all $ x\in B_{x_3}(t)$, $ y\in B_{x_1}(t)$ and $ z\in B_{x_2}(t)$ in this case.
Then, by
\eqref{T1}, \eqref{cz-size} and \eqref{Z10}, we have
\begin{align*}
\big|\jb{T(\phi_1^{x_1, t}, \phi_2^{x_2, t}), \phi_3^{x_3, t})}\big|
& \les t^{-2n} \iiint
| \phi_1^{x_1, t}(y) \phi_2^{x_2, t}(z)
\phi^{x_3, t}_3(x)| dy dz dx\\
& \les t^{-2n}
\prod_{j = 1}^3 \| \phi_j^{x_j, t}\|_{L^1}
\les t^n.
\end{align*}

\noi
Hence, \eqref{Z9} holds in both cases, thus completing the proof of the lemma.
\end{proof}

Now, we are ready to prove the weak boundedness property of the commutators.

\begin{lemma}
\label{L7}
Let $T \in \mathcal{O}p BS^1_{1, 0}$
and $a$ be a Lipschitz function.
Then, the bilinear commutators $[T, a]_j,$ $j=1, 2$, satisfy the weak boundedness property.
\end{lemma}

\begin{proof}
We only show that the weak boundedness property holds for $[T, a]_1$. A similar argument holds for $[T, a]_2$.
By Lemma \ref{L8}, it suffices to prove \eqref{X1}.
First, note that
we can assume that $a(x_0) = 0$, since
replacing $a$ by $a - a(x_0)$ does not change
the commutator.
Then, by the Fundamental Theorem of Calculus,  we have
\begin{align}
\|a\|_{L^\infty(B_{x_0}(t))} \les t\|\nabla a\|_{L^\infty}.
\label{X2}
\end{align}

\noi
By writing
\begin{align*}
\big|\jb{[T, a]_1& (\phi_1^{x_0, t},  \phi_2^{x_0, t}), \phi_3^{x_0, t})}\big|
\\
& \leq
\big|\jb{T(a\phi_1^{x_0, t}, \phi_2^{x_0, t}), \phi_3^{x_0, t})}\big|
+  \big|\jb{aT(\phi_1^{x_0, t}, \phi_2^{x_0, t}), \phi_3^{x_0, t})}\big|
=: \I + \II,
\end{align*}

\noi
it suffices to estimate $\I$ and $\II$ separately.

First, we estimate $\II$.
By \eqref{Z10}, \eqref{X2}  and  Lemma~\ref{L1}, we have
\begin{align*}
\II
& \leq   \|aT(\phi_1^{x_0, t}, \phi_2^{x_0, t}) \|_{L^2(B_{x_0}(t))}
\| \phi_3^{x_0, t}\|_{L^2} \notag \\
& \les t^{\frac n2}
 \|a\|_{L^\infty(B_{x_0}(t))}
 \| T(\phi_1^{x_0, t}, \phi_2^{x_0, t}) \|_{L^2(B_{x_0}(t))}\notag \\
 & \les t^{\frac n2 + 1}
 \|\nabla a\|_{L^\infty}
 \bigg\| \sum_{j = 1}^n
 \big[ T_j (D_j \phi_1^{x_0, t}, \phi_2^{x_0, t})
 + \wt {T}_j ( \phi_1^{x_0, t},D_j \phi_2^{x_0, t})\big]
\bigg\|_{L^2}\notag
\end{align*}

\noi
By the fact that $T_j, \wt T_j \in \mathcal{O}p\, BS^0_{1, 0}$ and \eqref{Z10}, we have
\begin{align*}
\II & \les t^{\frac n2 + 1}
 \|\nabla a\|_{L^\infty}
\sum_{j = 1}^n \Big[
\| D_j \phi_1^{x_0, t}\|_{L^4}\| \phi_2^{x_0, t}\|_{L^4}
 + \| \phi_1^{x_0, t}\|_{L^4}\|D_j \phi_2^{x_0, t}\|_{L^4} \Big]\notag \\
& 
\les t^{n}\|\nabla a\|_{L^\infty}.
\end{align*}

Next, we estimate $\I$. As before, by  Lemma~\ref{L1}, \eqref{Z10}
and \eqref{X2}, we have
\begin{align*}
\I
& \les t^\frac{n}{2}
 \bigg\| \sum_{j = 1}^n
 \big[ T_j (D_j (a \phi_1^{x_0, t}), \phi_2^{x_0, t})
 + \wt {T}_j ( a \phi_1^{x_0, t},D_j \phi_2^{x_0, t})\big]
\bigg\|_{L^2}  \notag \\
& \les t^\frac{n}{2}
 \sum_{j = 1}^n
\Big[
\|D_j (a \phi_1^{x_0, t})\|_{L^4}\| \phi_2^{x_0, t}\|_{L^4}
 + \| a \phi_1^{x_0, t}\|_{L^4} \|D_j \phi_2^{x_0, t}\|_{L^4}
\Big] \notag \\
& \les t^\frac{n}{2}
 \sum_{j = 1}^n
\Big[
\|D_j (a) \phi_1^{x_0, t}\|_{L^4}\| \phi_2^{x_0, t}\|_{L^4}
+ \| a D_j\phi_1^{x_0, t}\|_{L^4}\| \phi_2^{x_0, t}\|_{L^4}\notag \\
& \hphantom{XXXXX}
 + \| a \phi_1^{x_0, t}\|_{L^4} \|D_j \phi_2^{x_0, t}\|_{L^4}
\Big]\notag \\
& \les t^\frac{n}{2}
 \sum_{j = 1}^n
\Big[
t ^{\frac n2} \|\nabla a\|_{L^\infty}
+ t^{\frac{n}{2}-1} \|a\|_{L^\infty(B_{x_0}(t))}
\Big]
 \les t^n  \|\nabla a\|_{L^\infty}.
\end{align*}

\noi
This completes the proof of Lemma \ref{L7} and thus the proof of Theorem~\ref{main1}.
\end{proof}

\noindent {\bf Remark.} We wish to end this work by observing that the converse of Theorem~\ref{main1} also holds.
Let $T_j \in \mathcal{O}pBS^1_{1, 0}$, $j = 1, \dots, n$,  be defined by $T_j (f, g)=(D_j f)g$.
Suppose that $[T_j, a]_1$ is bounded  from $L^4\times L^4$ into $L^2$,
$j = 1, \dots, n$.
Then, $a$ is a Lipschitz function.
See Theorem~A for the converse statement in the linear setting.

The proof is immediate.
Noting that  $[T_j, a]_1 (f, g)=(D_j a)fg$,
 the boundedness of $[T_j, a]_1$ then forces $D_j a\in L^\infty$
(say, by taking $f = g$ to be a bump function localized near the maximum of $D_j a$).
 Since this is true for all $1\leq j\leq n$, $a$ must be Lipschitz.

 In particular,
 if we assume that $[T, a]_1$ is bounded from $L^4\times L^4$ into $L^2$
for all $T\in \mathcal{O}p\, BS^1_{1, 0}$, then $a$ must be a Lipschitz function. Of course, the boundedness $[T, a]_1: L^4\times L^4\to L^2$ can be exchanged with a more general one $L^p\times L^q\to L^r$ for some H\"older triple $(p, q, r)\in [1, \infty)^3$.
 An analogous statement applies to the second commutator $[T, a]_2$.


\begin{thebibliography}{50}

\bibitem{BMNT} \'A.\ B\'enyi, D.\ Maldonado, V.\ Naibo, and R.\ H.\ Torres, \emph{On the H\"ormander classes of bilinear pseudodifferential operators,} Integral Equations Operator Theory {\bf 67} (2010), no. 3, 341-364.

\bibitem{BBMNT} \'A. B\'enyi, F.\ Bernicot, D.\ Maldonado, V.\ Naibo, and R.\ H.\ Torres, \emph{On the H\"ormander classes of bilinear pseudodifferential operators, II,} Indiana Univ. Math. J., to appear;
preprint at arXiv:1112.0486 [math.CA].

\bibitem{BNT} \'A.\ B\'enyi, A.\ R.\ Nahmod, and R.\ H.\ Torres, \emph{Sobolev space estimates and symbolic calculus for bilinear pseudodifferential operators,} J. Geom. Anal. {\bf 16} (2006), no. 3, 431-453.

\bibitem{BT} \'A.\ B\'enyi and R.\ H.\ Torres, \emph{Symbolic calculus and the transposes of bilinear pseudodifferential operators,} Comm.\ P.D.E. {\bf 28} (2003), 1161-1181.

\bibitem{BT1} \'A. B\'enyi and R.\ H.\ Torres, {\it Compact bilinear operators and commutators,} Proc. Amer. Math. Soc., to appear.

\bibitem{BMMN} F.\ Bernicot, D.\ Maldonado, K.\ Moen, and V.\ Naibo, \emph{Bilinear Sobolev-Poincar\'e inequalities and Leibniz-type rules,} J. Geom. Anal., to appear;
preprint at  arXiv:1104.3942 [math.CA].

\bibitem{Cal1} A.\ P.\ Calder\'on, \emph{Algebras of singular integral operators,} AMS Proc. Symp. Math. {\bf 10} (1966), 18-55.

\bibitem{Cal2} A.\ P.\ Calder\'on, \emph{Commutators of singular integral operators}, Proc. Nat. Acad. Sci. {\bf 53} (1965), 1092-1099.

\bibitem{Cal3} A. P. Calder\'on, \emph{Intermediate spaces and interpolation, the complex method}, Studia Math. 24 (1964), 113-190.

\bibitem{CLMS} R.\ R.\ Coifman, P.\ L.\ Lions, Y.\ Meyer, and S.\ Semmes, \emph{Compensated compactness and Hardy spaces,} J.\ Math.\ Pures Appl. {\bf 72} (1993), 247-286.

\bibitem{CRW} R.\ Coifman, R.\ Rochberg, and G.\ Weiss, {\it Factorization theorems for Hardy spaces in several variables}, Ann. of  Math. {\bf 103} (1976), 611-635.

\bibitem{CM} R.\ R.\ Coifman and Y.\ Meyer, \emph{Au-del\`a des Op\'erateurs Pseudo-diffe\'rentiels,} Ast\'erisque {\bf 57}, Soci\'et\'e Math. de France, 1978.

\bibitem{CW} R. Coifman and G. Weiss, {\it Extension of Hardy spaces and their use in analysis}, Bull. Amer. Math. Soc. {\bf 83} (1977), 569-645.

\bibitem{DJ} G.\ David and J.\ L.\ Journ\'e, \emph{A boundedness criterion for generalized Calder\'on-Zygmund operators,} Ann. of Math. {\bf 120} (1984), 371-397.

\bibitem{GT} L.\ Grafakos and R.\ H.\ Torres, \emph{Multilinear Calder\'{o}n-Zygmund theory,} Adv.\ Math. {\bf 165} (2002), 124-164.

\bibitem{GT1} L.\ Grafakos and R.\ H.\ Torres, \emph{Maximal operator and weighted norm inequalities for multilinear singular integrals}, Indiana Univ. Math. J. {\bf 51} (2002), no. 5, 1261-1276.

\bibitem{H} J.\ Hart, \emph{A new proof of the bilinear $T(1)$ theorem,}     Proc. Amer. Math. Soc., to appear.

\bibitem{Iwa} T.\ Iwaniec, \emph{Nonlinear commutators and Jacobians,}
J. Fourier Anal. Appl. {\bf 3} (1997), 775-796.

\bibitem{IS} T.\ Iwaniec and C.\ Sbordone, \emph{Riesz transform and elliptic PDEs with VMO coefficients,} J. Anal. Math. {\bf 74} (1998), 183-212.

\bibitem{KP} T.\ Kato and G.\ Ponce, \emph{Commutator estimates and the Euler and Navier-Stokes equations,} Comm. \ Pure \ Appl. \ Math. {\bf 41} (1988), 891-907.

\bibitem{KPV} C.\ Kenig, G.\ Ponce and L.\ Vega, \emph{On unique continuation for nonlinear Schr\"odinger equations,} Comm.\ Pure\ Appl. Math.\ {\bf 56} (2003), 1247-1262.

\bibitem{MC} Y.\ Meyer and R.\ R.\ Coifman, \emph{Wavelets: Calder\'on-Zygmund and Multilinear Operators,} Cambridge University Press, Cambridge, United Kingdom, 1997.

\bibitem{MT} A. Miyachi and N. Tomita, \emph{Calder\'on-Vaillancourt type theorem for bilinear pseudo-differential operators,} Indiana Univ. Math. J., to appear.

\bibitem{PT} C.\ P\'erez and R.\ H.\ Torres, {\it Sharp maximal function estimates for multilinear singular integrals}, Contemp. Math. {\bf 320} (2003), 323-331.

\bibitem{PPTT} C. P\'erez, G. Pradolini, R. H. Torres, and R. Trujillo-Gonz\'alez, \emph{End-points estimates for iterated commutators of multilinear singular integrals}, Bull. Lond. Math. Soc., to appear; preprint at arXiv:1004.4976 [math.CA].

\bibitem{S} E. Stein, \emph{Harmonic analysis: Real variable methods, orthogonality, and oscillatory integrals.} Princeton University Press, Princeton, New Jersey, {\bf 1993}; 695 pp.

\bibitem{Tan} L.\ Tang, \emph{Weighted estimates for vector-valued commutators of multilinear operators},  Proc. Roy. Soc. Edinburgh Sect. A {\bf 138} (2008), 897-922.

\bibitem{Uch} A.\ Uchiyama, \emph{On the compactness of operators of Hankel type},  T\^ohoku Math. J. {\bf 30} (1978), no. 1, 163-171.

\end{thebibliography}
\end{document}